\documentclass[a4paper]{article}
\usepackage{amsthm,amssymb,amsmath,enumerate}
\usepackage{algorithm}
\usepackage{subcaption}
\usepackage{csvsimple}

\usepackage{eurosym}

\usepackage{tikz}
\usetikzlibrary{backgrounds}
\usetikzlibrary{positioning}
\usetikzlibrary{shapes}

\tikzstyle{hvertex}=[thick,circle,inner sep=0.cm, minimum size=2.3mm, fill=white, draw=black]
\tikzstyle{hedge}=[very thick]
\tikzstyle{rededge}=[very thick,red]
\tikzstyle{point}=[draw,circle,inner sep=0.cm, minimum size=1mm, fill=black]
\tikzstyle{pointer}=[thick,->,shorten >=2pt,color=hellgrau]
\tikzstyle{facebdry}=[color=auchblau, very thick] % face boundary
\tikzstyle{face}=[facebdry,fill=hellblau]
\tikzstyle{nface}=[color=hellblau,fill=hellblau,thick] % naked face, without boundary

\colorlet{auchblau}{blue!60!white}
\colorlet{hellblau}{blue!20!white}
\colorlet{hellrot}{red!40!white}
\colorlet{hellgrau}{black!30!white}
\colorlet{dunkelgrau}{black!60!white}

\newcommand{\slot}[2]{
\draw[ultra thick] (#1-0.5,#2-0.5) rectangle ++(1,1);
}

\tikzstyle{rline}=[ultra thick]

\def\robot{
\begin{scope}[scale=1.5]
\draw[line width=3pt,dunkelgrau, fill=white] (0,-0.5) circle [radius=0.3];
\draw[line width=10,white] (0,-0.5) -- ++(0,0.5);
\draw[rline] (-0.1,-0.5) -- ++(0,0.5);
\draw[rline] (0.1,-0.5) -- ++(0,0.5);
\draw[rline, fill=white] (0,-0.5) circle [radius=0.15];
\draw[line width=2pt] (0,0) arc [start angle=230, delta angle=-80, radius=0.3];
\draw[line width=2pt] (0,0) arc [start angle=310, delta angle=80, radius=0.3];
\draw[rline, fill=white] (0,0) circle [radius=0.1];
\end{scope}
}

\newtheorem{definition}{Definition}
\newtheorem{proposition}[definition]{Proposition}

\newtheorem{lemma}[definition]{Lemma}

\newtheorem{problem}[definition]{Problem}
\newtheorem{observation}[definition]{Observation}

\newcommand{\comment}[1]{}

\newcommand{\emtext}[1]{\text{\em #1}}

\newcommand{\prob}[1]{\textsf{#1}}

\usepackage{algpseudocode}

\usepackage{framed}
\usepackage{boxedminipage}

\newlength{\innerboxwidth}
\newenvironment{XXhproblem}[1]{
\smallskip\noindent
\begin{boxedminipage}{\textwidth}\vspace*{3pt}
\prob{#1}\\[6pt]
\hspace*{1.3cm}\begin{minipage}{\innerboxwidth}
\begin{itemize}\labelwidth3cm
}{
\end{itemize}\end{minipage}\vspace*{3pt}
\end{boxedminipage}
}

\newcommand{\mpty}{\sqcup}

\newcommand{\PName}{\textsc{Belt Makespan}}
\newcommand{\SR}{SR}
\newcommand{\IR}{NR}
\newcommand{\SRL}{SR+loc}
\newcommand{\IRL}{NR+loc}
\newcommand{\human}{\textsc{Human}}

\title{Case study on scheduling cyclic conveyor belts}
\author{Felix Bock and Henning Bruhn}
\date{}

\begin{document}

\maketitle

\begin{abstract}
We optimise   the production
line of a manufacturing company in southern Germany in order to improve throughput.
While the optimisation problem is NP-hard in general,
analysing production data we 
find that  
in practice the problem can be 
solved very efficiently by  aggressive generation of random machine schedules.  
\end{abstract}

\section{Introduction}
We evaluate algorithmic strategies to improve throughput in the production
line of a manufacturing company in southern Germany.%
\footnote{
To protect its trade secrets, the company does not want to be named nor does it want us 
 state what it is producing.
} 
The company has a yearly revenue in excess of \euro 50 million, more than 500 employees and 
is among the world leaders in its market segment. It is mainly focused
on a single  product, which, however, is sold in a large number of variants
that vary by colour, materials, size and other characteristics. 
Due to the similarity of the variants, most of them can be manufactured with largely the same 
production process. 

One of the first steps in this process
consists of an automatic production line that comprises
a circular conveyor belt, a number of industrial robots and a hydraulic press. 
The company reports that, especially at the end of production runs, the production line
is not always at full capacity. We investigate here how better machine scheduling leads to a more efficient
use of the production line. 

The product is made in moulds that circulate on the conveyor belt.
Whenever a mould reaches one of the industrial robots, 
the robot performs a number of steps, such as the application of certain materials. Once this is done 
the belt advances the mould to the next robot. The last step  (in this line) consists of the hydraulic press. 
When the mould leaves the press, the product is removed and  transported to other parts of
the factory, where it undergoes further production steps such as polishing and assembly.
The mould, if it is still needed, moves to the start of the line and continues circulating 
on the conveyor belt. If it is no longer needed, it will be ejected; see Figure~\ref{robofig} for 
an illustration. 

The whole production process is organised into \emph{jobs}. Each job specifies which \emph{types}
of products need to be made and in which quantities. That is, each job lists \emph{demands} per 
product type. The automatic production line runs until a job is completely finished; only then
the next job may be started. The objective is to maximise the throughput, or equivalently, to 
minimise the \emph{makespan}, the running time, of each job. 

\begin{figure}[ht]
\centering
\begin{tikzpicture}[scale=0.8]
\tikzstyle{form}=[regular polygon,regular polygon sides=8,very thick,draw,minimum size=0.7cm,color=dunkelgrau,text=black, fill=hellgrau,inner sep=0,rounded corners=1pt];
\tikzstyle{product}=[thin,circle,inner sep=0.cm, minimum size=1.5mm, fill=dunkelgrau, draw=dunkelgrau]

%\draw[fill=hellblau] (4.5,-0.5) rectangle ++(1,1);

\foreach \i in {0,...,5}{
  \slot{0}{\i}
  \slot{5}{\i}
}

\foreach \j in {1,...,4}{
  \slot{\j}{0}
  \slot{\j}{5}
}

% injection
\draw[ultra thick] (3.5,-2) to ++(0,1.5);
\draw[ultra thick, dunkelgrau,->] (4,-2.2) to ++(0,1);

% ejection
\draw[ultra thick] (4.5,-2) to ++(0,1.5);
\draw[ultra thick] (5.5,-2) to ++(0,1.5);
\draw[ultra thick, dunkelgrau,<-] (5,-2.2) to ++(0,1);

\begin{scope}[shift={(-1.3,2.5)},rotate=-140]
\robot
\end{scope}

%\begin{scope}[shift={(6.3,3)},rotate=45]
%\robot
%\end{scope}

\begin{scope}[shift={(6,1.6)}]
\draw[fill=dunkelgrau,dunkelgrau] (0.45,0.8) rectangle ++(0.3,0.2);
\draw[fill=dunkelgrau,dunkelgrau] (0.45,0.0) rectangle ++(0.3,-0.2);
\draw[fill=dunkelgrau,dunkelgrau] (0,0.5) rectangle ++(1.2,0.3);
\draw[fill=dunkelgrau,dunkelgrau] (0,0) rectangle ++(1.2,0.3);
\node[align=center] at (0.7,1.7) {hydraulic\\press};
\end{scope}

\begin{scope}[shift={(3,6.5)},rotate=-100]
\robot
\end{scope}

\draw[ultra thick, rounded corners=3pt,fill=hellgrau] (6.5,-0.5) rectangle ++(1.5,1);
\draw[ultra thick, dunkelgrau,->] (5.6,0) to (7,0);
\node[align=center] at (6.5+0.75,-1.2) {finished\\product};
\node[product] at (7.2,0.2) {};
\node[product] at (7.7,-0.2) {};
\node[product] at (7.6,0.3) {};

\node[form] at (0,1) {$A$};
\node[form] at (0,3) {$A$};
\node[form] at (0,5) {$B$};
\node[form] at (2,5) {$C$};
\node[form] at (4,5) {$B$};
\node[form] at (5,4) {$B$};
\node[form] at (5,2) {$C$};
\node[form] at (5,0) {$C$};
\node[form] at (3,0) {$C$};
\node[form] at (1,0) {$C$};

\begin{scope}[shift={(2.5,2.5)}]
\draw[ultra thick, dunkelgrau, ->] (200:1) arc [start angle=200, end angle=-130, radius=1];
\end{scope}

\node at (-1.8,4) {robot};
\node at (3.5,-2.5) {injection};
\node at (5.5,-2.5) {ejection};
\node[align=center] (M) at (-0.5,-1.5) {mould\\of type $C$};
\draw[pointer,out=0,in=-90] (M) to (1,-0.5);
\node[align=center] (B) at (6.3,6.5) {conveyor\\belt};
\draw[pointer,out=-90,in=40] (B) to (5.6,5);
\end{tikzpicture}
\caption{The set-up of the production line}\label{robofig}
\end{figure}
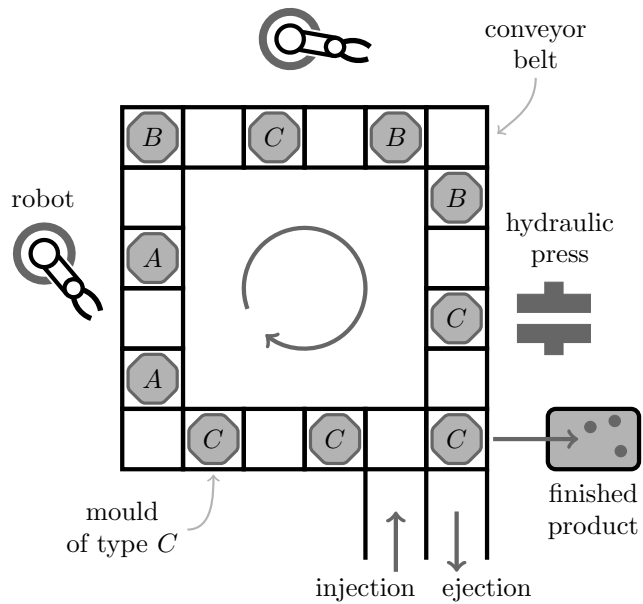

Different product types require different moulds. Of each type there may be a small number 
of moulds. That is, if a job specifies that product of type~$A$ and of type~$B$ should be 
produced then, at a given moment, perhaps three moulds of type~$A$ and two moulds of type~$B$ may
be circulating on the conveyor belt. Moulds are heavy and expensive, and consequently, the company
has always only a small number of moulds of each type, normally from one {up} to eight moulds. 
The conveyor belt is divided into a number of slots. 
Moulds may be injected at any time at the beginning of the line, provided 
the slot at the beginning is empty. Due to technical restrictions, 
injected moulds  stay
on the belt until the production of its type has reached its demand, at which stage all the 
moulds of the type are ejected. This results in empty slots on the belt that may be filled 
with {other moulds}. 

The length of the makespan of a job  depends on which moulds are placed on the belt and, 
to a somewhat lesser degree, in which order. Typically, the number of available moulds for a
job is larger than the number of slots on the belt so that not all of them can be placed
on the belt at the same time. 
In Section~\ref{modelsec}, we formulate a mathematical model, called \PName,
that captures how the placement of the moulds affects the production time of the job.
In Section~\ref{algosec}, we propose  simple and fast algorithms that we 
evaluate in Section~\ref{testsec}.

% Related work
\medskip

We give a very brief overview on the relevant literature. There are some 
known types of optimisation problems that are related to \PName. 
By treating each slot on the conveyor belt as a machine, for instance, we
may view it as a parallel machine scheduling problem. \PName \ is  located somewhere between parallel machine 
scheduling with splitting jobs, in which items may be processed by all machines simultaneously, and parallel
machine scheduling with preemption, in which items may be processed by at most one machine at a time.
Both of these problems can be solved efficiently~\cite{xing2000parallel,mcnaughton1959scheduling}.
\PName , however, is different: it turns out to be NP-hard (see Section~\ref{hardsec}).

Scheduling (linear) production lines 
 are often modelled as a permutation flow shop. 
In the special case of two machines the problem can be solved optimally in polynomial time 
with Johnson's rule~\cite{johnson1954optimal}.
For a larger number of machines the problem becomes hard~\cite{lenstra1977complexity}.
A key feature of our problem is the limited availability of moulds. In the context of 
scheduling problems such type of resources is often referred to as renewable resources.
Several algorithms have been proposed for 
scheduling problems with renewable resources, such as mixed integer programming~\cite{laribi2016heuristics},
genetic algorithms~\cite{mori1997genetic} and branch-and-bound strategies~\cite{stinson1978multiple}.
None of these algorithms fits our problem setting completely.

\section{The mathematical model}\label{modelsec}\label{hardsec}

We first turn the production process into an idealised mathematical model. 
To formulate the model we will make a number of assumptions that will, in practice, 
not always be satisfied. For example, we assume that the conveyor belt  advances
with constant speed. Production data, however, indicates that the moulds advance
faster along the production line if fewer moulds are placed on the belt. 
Nevertheless, we believe that our model captures the real production process quite 
well. We will discuss the validity of the model in Section~\ref{testsec}.

Let $N$ be a positive integer, the \emph{number of slots} of the conveyor belt.
Production data indicates that $N=20$ is a reasonable estimate.
We treat $N$ as a constant since it is, in fact, constant: it is a feature of the production process. 

An \emph{instance}, or job, then consists of a finite set $\mathcal T$, the set of (product) \emph{types},
as well as positive integers $d_A$ and $c_A$ for each type $A\in\mathcal T$.
The integer $d_A$ is the \emph{demand} of type $A$ and $c_A$ is the \emph{capacity} of type $A$, i.e.\ $c_A$
is the number of moulds of type~$A$ that are available.

Let $\mpty\notin\mathcal T$ be a symbol signifying an empty slot. 
Given an instance $(\mathcal T,(d_A)_{A\in \mathcal{T}},(c_A)_{A\in \mathcal{T}})$, we model the
 usage of the conveyor belt with a function $b:\mathbb Z_{\geq 0}\to \mathcal T\cup\{\mpty\}$
that  represents the order in which the moulds are circulating on the belt.
We say that (with respect to $b$, which we usually leave out)
a \emph{mould of type~$A$ is injected in step $i$} if $b(i)=A$ but $b(i-N)\neq A$, or if $i-N<0$.
(This models the first occurence of a mould of the type.)
We say that \emph{$k$ moulds of type $A$ are used} if $k$ is the number of steps in which a mould of type~$A$
is injected.\sloppy

For an integer $n$, we write $[n]$ for the set $\{1,2,\ldots,n\}$. 
A \emph{belt assignment} is a function $b:\mathbb Z_{\geq 0}\to \mathcal T\cup\{\mpty\}$ 
such that 
\begin{enumerate}[\rm (i)]
\item $|b^{-1}(A)|=d_A$ for every $A\in\mathcal T$; 
\item  if $b(i)=A$ for some non-negative integer $i$ and $A\in\mathcal T$ then $b(i+N)=A$ or 
$|b^{-1}(A)\cap [i+N-1]|=d_A$; and
\item for every $A\in\mathcal T$ the number of moulds of type~$A$ that are used is at most~$c_A$.
\end{enumerate}
A belt assignment models the contents of the slot at the mould feed-in; condition~(i) forces
every type to be produced according to  demand; condition~(ii) states that an injected mould stays
on the conveyor belt until all of its type is produced (and will be seen every $N$ steps
in the first slot); and condition~(iii) guarantees that for every type only as many moulds are
used as physically available.

The \emph{makespan} of $b$ is the smallest integer $s$ such that $|b^{-1}(A)\cap [s-N]|=d_A$ for every $A\in\mathcal T$.
(The reason for the $-N$ in $[s-N]$ lies in the fact that the production process, i.e.\ the work
the industrial robots and the hydraulic press perform, takes $N$ steps; that is, once the last mould
is seen on the first slot, the mould will continue for another $N$ steps before the job is finished.)
Now we can formally define the optimisation problem:

\begin{problem}[\PName]\ \\
Let a set $\mathcal T$ of types as well as demands
$(d_A)_{A \in \mathcal T}$ and capacities $(c_A)_{A \in \mathcal T}$ be given.
Find a belt assignment $b:\mathbb Z_{\geq 0}\to \mathcal T\cup\{\mpty\}$ that minimises the makespan.
\end{problem}

The problem is computationally hard:
\begin{proposition}\label{npprop}
\PName\ is NP-hard if $N\geq 2$.
\end{proposition} 

The proof is via reduction to \textsc{Partition}, a problem that is well-known to be NP-complete~\cite{garey2002computers}. 
\begin{problem}[\textsc{Partition}] \ \\
Given positive integers $d_1,d_2,\ldots,d_n$ decide
whether there is a subset $S\subseteq [n]$ such that $\sum_{i \in S} d_i = \sum_{i \in [n]\setminus S} d_i$.
\end{problem}

\begin{proof}[{Proof of Proposition~\ref{npprop}}]
We reduce \textsc{Partition} to \PName \  by
\begin{itemize}
\item  creating a type $i$ with demand $d_i$ and capacity $c_i=1$ for each number $d_i$ of \textsc{Partition}, and
\item adding $N-2$ types of demand $\frac{1}{2} \sum_{i \in [n]} d_i$ and capacity $1$.
(If $\sum_{i \in [n]} d_i$ is not even, the \textsc{Partition} instance is a \textsc{no}-instance anyway.)
\end{itemize}
Then, there is a belt assignment with makespan $N \cdot \frac{1}{2} \sum_{i \in [n]} d_i + (N-1)$ if and only if there is a set $S\subseteq [n]$ such that $\sum_{i \in S} d_i = \sum_{i \in [n]\setminus S} d_i$. 
% this is all a little bit fast...
\end{proof}

For later use, we prove a lower bound on the makespan.

\begin{lemma}
\label{lem:lowerBound}
Let $(\mathcal T,(d_A)_{A\in\mathcal T},(c_A)_{A\in\mathcal T})$ be an instance 
of  \PName\ and let $M^*$ be its optimal makespan.
Then  
\[
M^* \geq \max\left\{ \ \sum_{A \in \mathcal{T}} d_A , \,  (r^* -1) N + \sum_{A \in \mathcal{T}} \sigma_A \ \right\}+N-1,
\]
 where
$r^* := \max_{A \in \mathcal{T}}\Bigl\{\bigl\lceil \frac{d_A}{c_A} \bigr\rceil\Bigr\}$ 
and 
$\sigma_A := \max \left\{ 0,\, d_A - (r^*-1)c_A \right\} .$
\end{lemma}

\begin{proof}
Let $b$ be a belt assignment.
We can split the statement into two observations.

First, all of the items have to be produced. 
This requires $\sum_{A \in \mathcal{T}} d_A$ many steps $i$ with $b(i) \neq \mpty$. 
Together with the final $N-1$ steps, we get  $M^*\geq \ \sum_{A \in \mathcal{T}} d_A+N-1$.

Second, for non-negative $r$, let us call the subsequence 
\[
\Bigl( b\bigl(rN\bigr)\ ,\ b\bigl(rN+1\bigr)\ ,\ \ldots\ , \ b\bigl(rN+N-1\bigr) \Bigr)
\]
of $b$ the \emph{$(r+1)$-th round} of $b$.  
This corresponds to the moulds that are on the conveyor belt on its $(r+1)$-th rotation.
Every round of $b$ contains at most $c_A$ many moulds of any type $A$.
Thus, to satisfy the demand of $A$, at least $\lceil \frac{d_A}{c_A} \rceil-1$ full rounds
are needed, plus possibly part of round $\lceil \frac{d_A}{c_A} \rceil$. 
This already implies $M\geq (r^*-1)N+(N-1)$. 

Now, during the first $r^*-1$ many rounds at most $ (r^*-1)c_A $ many products of type~$A$
are made, which means that after the first $r^*-1$ many rounds, there is a remaining
demand of $\sigma_A$. Therefore, $M^* \geq (r^* -1) N + \sum_{A\in \mathcal{T}} \sigma_A + N-1$. 
\end{proof}

\bigskip

From now on we fix an instance $(\mathcal T,(d_A),(c_A))$ of \PName.
Note that we may assume that $c_A \leq N$ since at most $N$ moulds of type $A$ can be used anyway.
For a belt assignment $b$, we call a non-negative integer $e$ an \emph{empty slot} if $b(e)=\mpty$. 
The  key to  \PName\  is  an efficient encoding of belt assignments.
For this, we make a number of observations.

\begin{observation}\label{obs1}
If $b$ is a belt assignment with makespan $M$ then there is a belt assignment $b'$ with makespan 
at most $M$ with the following property:
\begin{equation}\label{econom}
\text{
If a mould is injected in  step $i$ and if $e$ is an empty slot then $e>i$. 
}
\end{equation}
\end{observation}
\begin{proof}
Suppose that the statement is not true and choose $b$  among
all $b$ such that there is no $b'$ as in the statement so that the earliest empty slot $e$ is as
late as possible, i.e.\ so that $e$ is maximal.

Suppose first that there is an $i>e$ such that a mould is injected in step $i$ and 
such that $i=e+\ell N$, for some positive integer $\ell$.
Let $i^*$ be chosen minimal with respect to this property and let $A^*=b(i^*)$. 
Define a new belt assignment $b'$ by starting with
\[
b''(k)=\begin{cases}
A^* & \textrm{if $k=e+r N$ for $r=0,1,\ldots, \ell-1$}\\
b(k) & \textrm{else}
\end{cases}
\]
and removing the last $\ell$ appearances of type $A^*$ until there are only $d_{A^*}$ many left. 
By construction, $b'$ is  a  belt assignment with a makespan no greater than that of $b$,  
 and because $b'(e) \neq \mpty$  its earliest empty slot is 
strictly later than $e$.

Suppose now that no injection occurs in a step $e+\ell N$, for positive $\ell$. That means that 
all $e,e+N,e+2N,\ldots$ are empty slots.
Pick an $i^*>e$, in which a mould is injected, and let $A^*=b(i^*)$. Define $b'$ by starting with
\[
b''(k)=\begin{cases}
A^*&\textrm{if $k=e+\ell N$ for $\ell=0,1,\ldots, d_{A^*}$}\\
\mpty & \textrm{if $k=i^*+\ell N$ for $\ell=0,1,\ldots, d_{A^*}$ and $b(k)=A^*$}\\
b(k)&\textrm{else}
\end{cases}
\]
and removing the last appearances of type $A^*$ until there are only $d_{A^*}$ many left. 
By construction, $b'$ is again a  belt assignment and  its earliest empty slot is strictly 
later than $e$.
\end{proof}

\begin{observation}\label{obs2}
If $b$ is a belt assignment with makespan $M$ then there is a belt assignment $b'$ with makespan 
at most $M$ with property \eqref{econom} and the additional property:
\begin{equation}\label{eq:moreMoulds}
\begin{minipage}[c]{0.8\textwidth}
\textrm{If $e$ is an empty slot and $A \in \mathcal{T}$ is a type for 
which in $b'$ strictly fewer than $c_A$  moulds are used,
then $b'(i) \neq A$  for all $i \geq e$.   
}
\end{minipage}\ignorespacesafterend 
\end{equation} 
\end{observation}
What this means: If the demand of type~$A$ is not yet satisfied and an empty slot
comes up, then we may put a mould onto the belt, provided there is still one available.
\begin{proof}
Suppose the statement to be false. By Observation~\ref{obs1}, we may assume $b$ to satisfy~\eqref{econom}. 
Among all $b$ with~\eqref{econom} that violate~\eqref{eq:moreMoulds} choose one in which 
the earliest empty slot $e$ is as
late as possible, i.e.\ so that $e$ is maximal.

Let $A$ be a type for which \eqref{eq:moreMoulds} is false.
Note that~\eqref{econom} implies that empty slots stay empty, that is, $b(e)=\mpty$ implies
$b(e+\ell N)=\mpty$ for all integers $\ell\geq 0$. 
Again, define $b'$ by starting with
\[
b''(k)=\begin{cases}
A&\textrm{if $k=e+\ell N$ for $\ell=0,1,\ldots,d_A$}\\
b(k)&\textrm{else}
\end{cases}
\]
and removing the last appearances of type $A$ until there are only $d_{A}$ many left. 
By construction, $b'$ uses exactly one more mould of type $A$ than $b$ but still at most~$c_A$. 
In $b$ all moulds were injected before step $e$, by~\eqref{econom}. Because of that
in $b'$ the last mould is inserted in step $e$ and there is no empty slot before it. Thus, property \eqref{econom} still holds.
Moreover, the makespan of $b'$ is as most $M$ and it has no empty slot among the first $e$ slots. This contradicts the choice of $b$.
\end{proof}

In view of Observations~\ref{obs1} and~\ref{obs2},
we may restrict ourselves to belt assignments that satisfy \eqref{econom} and \eqref{eq:moreMoulds}.
Such assignments  allow an efficient encoding.
Given a belt assignment $b$ with~\eqref{econom} and~\eqref{eq:moreMoulds}, 
let $i_1<i_2<\ldots<i_D$ be the steps in which some mould is injected. 
Then the \emph{short injection sequence} of $b$ is $s=(b(i_1),b(i_2),\ldots,b(i_D))$. 

Given a short injection sequence $s$ of a belt assignment $b$, the belt assignment can be reconstructed by iteratively
placing the next mould of $s$ on the first empty slot of $b$.
This can be done with the following procedure:

\begin{algorithm}[htb]
\emtext{input:} injection sequence $s$
\begin{algorithmic}[1]
\State set $b(i)= \mpty$ for every $i\in \mathbb{Z}_{\geq 0}$
\State set $i=0$, $j=0$
\While{total demand not sat'd, i.e.\ $\sum_{A}|b^{-1}(A)|<\sum_Ad_A$}
\State if $i\geq N$  set $A=b(i-N)$, and $A=\mpty$ otherwise
\While{$A=\mpty$ or $|b^{-1}(A)|=d_A$}
\If{$j >$ length($s$)}
\State exit loop with $A=\mpty$.
\Else
\State set $A=s(j)$ and $j=j+1$
\EndIf
\EndWhile
\State set $b(i)=A$ and $i=i+1$
\EndWhile
\State return $b$
\end{algorithmic} \caption{(short) injection sequence to belt assignment} \label{algo:beltAssignment}
\end{algorithm}

Short injection sequences, however, are not easy to generately randomly. We therefore generalise the notion
to \emph{injection sequences}: any sequences of  types in $\mathcal{T}$ that contains 
every type $A$ (think mould)
exactly $c_A$ times. The procedure above also turns (non-short) injection sequences 
into belt assignments. 
The difference lies in the fact
that injection sequences may contain unnecessary moulds, moulds that can be omitted without changing
the belt assignment. 

The main point about injection sequences is: they
are easy to generate (just take a permutation of the sequence with all moulds) and still they encode all 
relevant belt assignments.

\begin{observation}
\label{obs:correspondingBeltAssignments}
Let $b$ be a belt assignment. Then  $b$ satisfies \eqref{econom} and \eqref{eq:moreMoulds} if and only if there is 
an injection sequence $s$ for which $b$ is the corresponding belt assignment.
\end{observation}

Belt assignments that are constructed by injection sequences even provide an approximation guarantee:

\begin{observation}
\label{obs:2factor}
The corresponding belt assignment of an injection sequence is a $2$-factor approximation for the \PName\ problem.
\end{observation}
\begin{proof}
Let $s$ be an injection sequence and $M(s)$ be the makespan of its corresponding belt assignment $b$.
If $M(s)=\sum_{A \in \mathcal{T}} d_A + N - 1$ then, by Lemma~\ref{lem:lowerBound},
 $b$ is an optimal belt assignment. 
Otherwise there is an empty slot $e$ with $e \leq \sum_{A \in \mathcal{T}} d_A -1$.
We note that, by Observation~\ref{obs:correspondingBeltAssignments}, the assignment $b$ satisfies~\eqref{econom} 
and~\eqref{eq:moreMoulds}.
%By \eqref{econom}, every injection has occurred before step $e$.

Let $A^*$ be the type that finishes last, i.e.\ $A^*:= b(M(s)-N)$. 
By~\eqref{eq:moreMoulds}, all of the $c_{A^*}$ moulds of type $A^*$ have already been injected before step $e$. 
Each injection requires an appearance of $A^*$. 
Thus, there are at most $d_{A^*}-c_{A^*}$ many appearance of $A^*$ after step $e$.
Moreover, since all moulds are injected, any $N$ consecutive steps between $e$ and $M(s)-N$ contain exactly $c_{A^*}$ appearances of $A^*$.
This provides an upper bound for the number of steps after $e$:
 $$\Big(M(s)-N\Big)-e \leq \left\lceil \frac{d_{A^*}-c_{A^*}}{c_{A^*}} \right\rceil N$$
Using the bound on $e$, we get
$$M(s)\leq \left\lceil \frac{d_{A^*}-c_{A^*}}{c_{A^*}} \right\rceil N + e + N
\leq \left(\left\lceil \frac{d_{A^*}}{c_{A^*}} \right\rceil-1\right)N +\sum_{A \in \mathcal{T}}  d_A + N-1 .$$
By Lemma~\ref{lem:lowerBound}, this is at most twice the optimal makespan.
\end{proof}

Observation \ref{obs:2factor} shows that it is easy to provide an approximation guarantee for \PName . 
However, a makespan that is twice as large as necessary would be a inacceptable in practice. 
The algorithms that are presented in the next section 
typically produce better results than this worst case guarantee.

\section{Algorithms}\label{algosec}

We formulate four simple algorithms for \PName . We have observed that injection sequences determine
belt assignments, that is, determine a run of the production line. Injection sequences, 
moreover, are easy to generate.
In our first algorithm that we call \emph{Simple Random Algorithm},
or \SR\ for short, we generate uniformly at random  injection sequences until some stopping criterion
is satisfied (maximum number of sequences or allotted running time reached).
The best sequence is returned.

\begin{algorithm}[htb]
\emtext{input:} \PName-instance
\begin{algorithmic}[1]
\State Let $s$ be an arbitrary injection sequence that contains precisely $c_A$ moulds of every type $A\in\mathcal T$ 
\State Choose permutations of $s$ uniformly at random until stopping criterion is satisfied.
\State Return sequence with smallest makespan.
\end{algorithmic}
\caption{Simple Random Algorithm (\SR)}
\end{algorithm}

For the next algorithm, \emph{Non-uniform Random Algorithm} (\IR), we replace the uniform 
probability distribution by one that is guided by the instance. Clearly, types of large demand 
require more urgency, as do types with very few moulds, i.e.\ with small capacity $c_A$.
Indeed, for a type with only a single mould available it seems important that we put
the single mould on the belt as early as possible. To reconcile these two aspects, demand and  
capacity, we  give more weight in the probability distribution to types with large 
ratio $d_A/c_A$.

\begin{algorithm}[htb]
\emtext{input:} \PName-instance
\begin{algorithmic}[1]
\Function{SequenceGenerator}{}
\State $\emtext{total}=\sum_{B \in \mathcal{T}} \frac{d_B}{ c_B}$
\For{ $k=1,\ldots,\sum c_A$}
\State Choose type $A^*$ with probability $P(A)=\frac{d_A}{c_A \cdot \emtext{total}}$  if $A$ \\ \qquad has still at least one mould left and 0 otherwise.
\State $s(k)=A^*$
\If{ this was the last mould of $A^*$}
\State $\emtext{total} = \emtext{total} -\frac{d_{A^*}}{c_{A^*}}$
\EndIf
\EndFor
\State Return $s$
\EndFunction
\State Draw random sequences with \textsc{SequenceGenerator}  until stopping criterion is satisfied.
\State Return sequence with smallest makespan.
\end{algorithmic}\caption{Non-uniform Random Algorithm (\IR)}
\end{algorithm} 

Next, we combine the two algorithms with a simple local search routine, see Algorithm~\ref{localgo}. 
We fix two parameters \textsc{Steps} and \textsc{Swaps}. Given an initial injection sequence, we then 
perform \textsc{Steps} many steps. In each step, we pick among \textsc{Swaps} many random swaps 
of two moulds in the injection sequence the one that results in the smallest makespan. (Note that
this makespan may be worse than the one of the initial sequence.)  
Combining \SR\ with local search (\SRL) then amounts to repeatedly generating an injection sequence
uniformly at random followed by a local search. The combination of \IR\ with local search (\IRL)
is obtained in the same way. 

\begin{algorithm}[htb]
\emtext{input:} \PName-instance
\begin{algorithmic}[1]
\Function{Local Search}{initial sequence}
\State Set current sequence to initial sequence.
\For{ $k=1,\ldots, \textsc{Steps}$}
\For{ $t=1,\ldots, \textsc{Swaps}$}
\State Pick $a,b$ in current sequence uniformly at random.
\State Swap $a,b$ in current sequence and compute makespan.
\State Reverse swap.
\EndFor
\State Set current sequence to best sequence in inner loop.
\EndFor
\State Return best sequence among all visited sequences.
\EndFunction
\State Let $s$ be an arbitrary injection sequence that contains precisely $c_A$ moulds of every type $A\in\mathcal T$ 
\While{stopping criterion not satisfied}
\State Choose permutations of $s$ uniformly at random as initial sequence. 
\State Execute \textsc{Local Search}(initial sequence)
\EndWhile
\State Return sequence with smallest makespan among all visited sequences.
\end{algorithmic}\caption{\SRL}\label{localgo}
\end{algorithm}

\section{Evaluation of the algorithms}\label{testsec}

% DATA

The manufacturing company supplied us with production data covering a period of roughly 6 
months with $349$ individual jobs and a total demand of about $250\,000$.
For each job the data lists demands, capacities (number of available moulds per type), 
start and end time of production as well as
time stamps: every time a mould exits the press its unique identifier, its type and the current time 
are recorded. In particular, the time stamps allow to reconstruct the order in which the moulds
were injected on the conveyor belt. 

An inspection of the data revealed several inconsistencies. 
Sometimes the number of produced items did not quite match the demand.
In other jobs,  more moulds were used than available, at least according to stored mould capacities.
We handled these inconsistencies by taking the number of actually produced items as demand 
and by  increasing the capacities whenever the data showed more moulds were used. 
We also uncovered some aspects that were not captured by our mathematical model:

%An inspection of the data revealed several aspects that are not covered by our mathematical model:

\begin{itemize}
\item The relative order of two moulds on the belt sometimes changed during a production run.
This is likely caused by the occasional removal of a mould. Indeed, 
sometimes moulds need to be cleaned as 
excess material  tends to accumulate in the moulds.
\item The speed with which the moulds advance depends mildly on 
the total load of the line.
\item The number of slots is not a constant. It is possible to put more than $N=20$ moulds
on the belt. This is, however, normally avoided as it results in substantially more wear and tear 
in a part of the conveyor belt.  
\end{itemize}

To verify whether our model nevertheless fits the data well enough we compare its predictions
to the actual production times. It turns out, however, that the difference between end and start 
time of a job is not  a good measure for the time spent fulfilling the demands 
of the job. In many jobs the time stamp data show extensive time periods during which the line
was  stopped. In some cases this was obviously due to a Sunday or bank holiday
between start and end of the job; in other cases machine errors and repairs are the likely causes. 
To correct for idle times, we identified time periods of ten minutes or more during which no 
time stamp was registered. Subtracting these idle times from the difference between end and start time
of a job yields the \emph{adjusted production time} that we compare to the \emph{computed makespan} of a job: 
the makespan our model returns when input the short injection sequence gleaned from the time stamp 
data (i.e.\ moulds in order of first appearance).

Figure~\ref{fig:validation} shows  adjusted production time versus computed makespan for each job.
While it is not a perfect fit, a clear linear relationship (linear least square fit in red) can be 
observed. 
We take this as  evidence that our model is fairly reasonable.

\begin{figure}[htb]
\begin{center}
\includegraphics[width=0.6\textwidth]{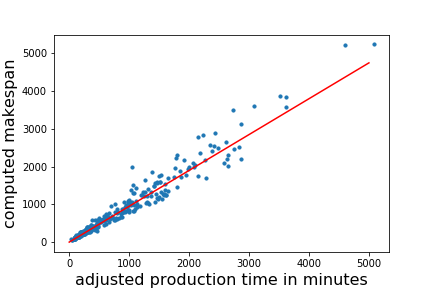}
\end{center}
\caption{Model validation. Each point represents a job. Linear least squares fit in red.}
\label{fig:validation}
\end{figure}

% Hardware & Parameter including time

To evaluate the algorithms, we implemented \SR, \IR, \SRL\ and \IRL\ in python and 
executed them on a 2018 standard desktop computer. We set the local search parameters
\textsc{Steps} and \textsc{Swaps} each to~$10$. We furthermore set a time limit
of~$1$~second per job for each algorithm. Every run of an algorithm was repeated~$10$ times.

Ideally, we would like to compare the performance of the different algorithms to 
the optimal makespan, the smallest makespan possible for each job, and 
indeed, it is straightforward to express \PName\ as a mixed integer program. 
We found, however, that at least with open source solvers and on standard hardware
the solvers did not finish in reasonable time. Instead, we use therefore the lower 
bound of Lemma~\ref{lem:lowerBound} as a stand-in for the optimum. 
While it is not hard to construct instances where the lower bound is off by a
factor arbitrarily close to~$2$, this is not a typical occurrence. In fact, 
we will see that, averaged over all instances in the production data, 
the factor between lower bound and the optimum is at most~$1.008$.

We also use the production data to evaluate the algorithms. Specifically, we compare 
the performance of the algorithms to the computed makespan defined earlier: 
the makespan that results from our model if the order of first appearance 
of the moulds is used as injection sequence. That sequence is decided by the workers
that service the production line. It is based on their experience and decided 
on, as far as we could see, mostly in a non-systematic intuitive way. 
In the table below, we denote the corresponding makespan with \human.

% Overall results -> table and overall figures

The overall test results are listed in Table~\ref{table:overallResult}. It shows 
makespans summed over all instances in the data set, averaged over the ten runs
of each algorithm.
The first row sets the human performance to~100\%, while the second row takes the lower bound of Lemma~\ref{lem:lowerBound}
as~100\%. 
With at most $0.05\%$ standard deviations were very small.

What can we conclude from these results? First, the human performance is actually quite good! 
Compared to the lower bound, the schedules prepared by the service workers leave only 
room for a reduction of at most 8.1\%, which would result in roughly 60 hours additional
production time per month. Second, we see that even our most basic algorithm, \SR, with 
a cumulated makespan of 101.7\% of the lower bound, comes
close to realising these potential savings in production time. \IR, with makespan 101.0\%
of the lower bound, comes closer still. Both algorithms benefit from local search, with \IRL\
saving fewer makespan steps than \SRL, most likely because the performance of \IR\ is already 
very close to the optimum. 
Considering that the lower bound is not tight in general,
we contend  that the performance of \IRL\ cannot be improved much, if at all.

\vspace*{-3pt}

\begin{table}[htb]
\begin{center}
\begin{tabular}{l|ccccc}
overall test result & \human & \SR & \SRL & \IR & \IRL  \\
 \hline
compared to \human   &  (100\%) & \ 93.5\% & \ 93.1\% & \ 92.8\% & \ 92.6\%   \\
compared to lower bound & 108.8\% &  101.7\% & 101.2\% & 101.0\% & 100.8\% 
\end{tabular}
\end{center}
\caption{Overall test result. First row: Cumulated makespans over all jobs, \human\ set to 100\%.
Second row: Cumulated makespans over all jobs, lower bound set to 100\%.}
\label{table:overallResult}
\end{table}

\vspace*{-3pt}
%\bigskip

In Figure~\ref{fig:overallSavings} the number of saved steps, the 
difference between \human\ makespan and the algorithmic makespan, is shown in histograms.
For each algorithm the number of saved steps is well distributed among many jobs.
Overall, the savings are not dominated by single outliers.

There are a few outliers, though, with many saved steps. 
It is likely that the poor performance of these extremes is not because of bad planning by the service
workers but due to external issues not captured by our model. 

In a large number of jobs, no algorithm can achieve a better schedule than the \human\ schedule. 
Inspecting  Figure~\ref{fig:overallGaps}, which shows the difference, or \emph{gap}, between 
the achieved makespan and the lower bound, reveals that a good number of the 
jobs was already solved optimally by the service workers. 
For comparison,
we also show the difference between the makespan achieved by \IRL\ and the lower bound.

\begin{figure}[h!]
\begin{center}
\begin{subfigure}[l]{0.45\textwidth}
\includegraphics[width=\textwidth]{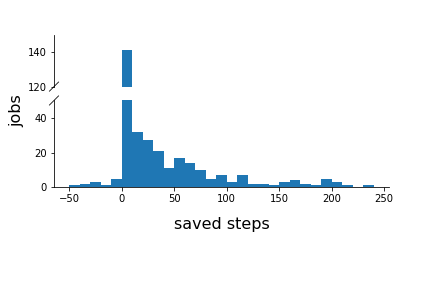}
\subcaption{Saved steps of \SR}
\end{subfigure}
\begin{subfigure}[l]{0.45\textwidth}
\includegraphics[width=\textwidth]{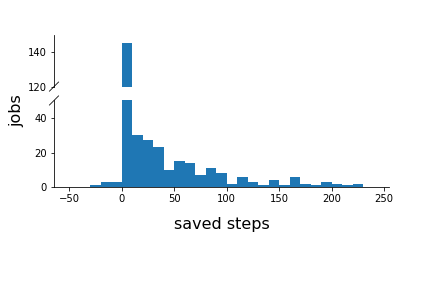}
\subcaption{Saved steps of \SRL}
\end{subfigure}\\
\begin{subfigure}[l]{0.45\textwidth}
\includegraphics[width=\textwidth]{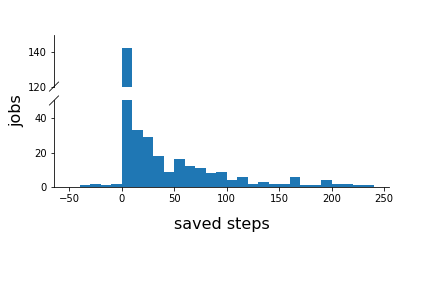}
\subcaption{Saved steps of \IR}
\end{subfigure}
\begin{subfigure}[l]{0.45\textwidth}
\includegraphics[width=\textwidth]{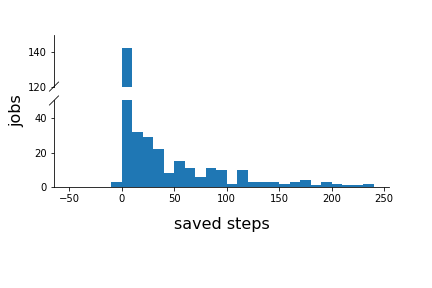}
\subcaption{Saved steps of \IRL}
\end{subfigure}
\end{center}
\caption{Distribution of saved steps.}
\label{fig:overallSavings}
\end{figure}

\begin{figure}[h!]
\begin{center}
\begin{subfigure}[l]{0.45\textwidth}
\includegraphics[width=\textwidth]{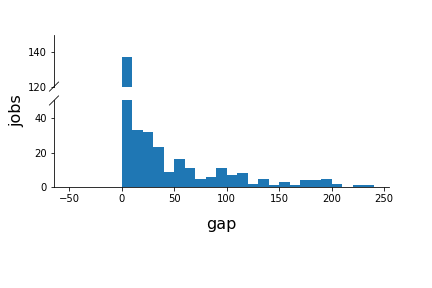}
\subcaption{Gap of \human\ makespan }
\end{subfigure}
\begin{subfigure}[l]{0.45\textwidth}
\includegraphics[width=\textwidth]{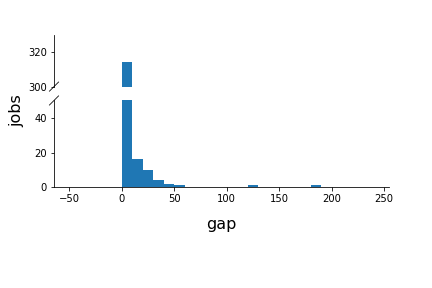}
\subcaption{Gap of  \IRL}
\end{subfigure}
\end{center}
\caption{Distribution of gaps to lower bound.}
\label{fig:overallGaps}
\end{figure}

Finally, we  analyse two typical jobs in more detail. The first, Job~1,
had a total demand of 1311, for which in sum 27 moulds were available. The second 
job, Job~2, had a total demand of~239 for which in total 18 moulds were available. 
Figure~\ref{fig:Loads} shows the \emph{loads}, the number of moulds on the belt,
in each \emph{round} of the algorithm, where we take a round as 
a collection of $N=20$ consecutive steps. 

The manufacturing company reported
that in particular during the end of many jobs the load of the production line was 
very low, that is, that only a few moulds were circulating along the line. 
The load diagram for Job~1 confirms this observation: the load of the \human\ schedule
drops steeply at round 58, resulting in a large makespan. Each of the algorithms,
on the other hand, manages to maintain a higher load for much longer, and thus achieves
a shorter makespan.

Job~2, on the other hand, is an example, where no improvement is possible. 
In total, fewer than 20 moulds are available; and thus  any schedule
that puts all available moulds on the belt is (almost) optimal. 
Still the schedule is inefficient: 
the load does not attain maximum capacity of the conveyor belt.

\begin{figure}
\begin{subfigure}[c]{0.5\textwidth}
\includegraphics[scale=0.4]{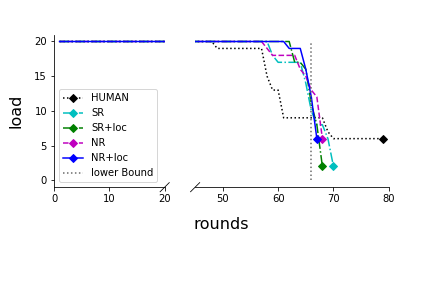}
\subcaption{Job 1}
\end{subfigure}
\begin{subfigure}[c]{0.5\textwidth}
\includegraphics[scale=0.4]{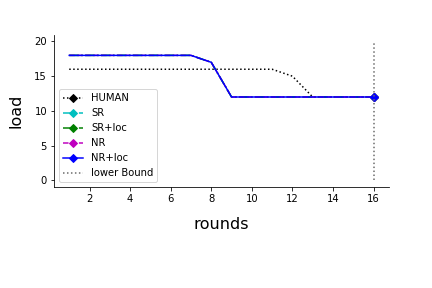}
\subcaption{Job 2}
\end{subfigure}
\caption{Load of conveyor belt.}
\label{fig:Loads}
\end{figure}

\section{Conclusion}

The starting point to this article was the observation of the manufacturing company that
their production line was not always operating at full, or nearly full, capacity. Did we succeed in 
elaborating a strategy to increase the load of the production line? 

We translated the mechanics of the production process into a mathematical optimisation problem
that, while hard from a complexity-theoretic point of view, can be solved very efficiently to close
to optimality. Our algorithms give moderate savings compared to the current, human made, schedules, 
at least when measured within the model. Close inspection of the results indicate that 
a good part of the savings should be realisable in practice. 
In this sense, we  succeeded in increasing the production load.

However, close inspection also reveals
that a substantial part of the jobs are solved optimally but still run inefficiently. 
These are jobs with too few moulds available (such as Job~2 above). 
Consequently, larger productivity gains may only 
be obtained by starting earlier with the optimisation, namely at the planning phase of the jobs, 
when the production demands over a longer period of time are assessed and assembled into jobs.

\bibliography{references}

\providecommand{\bysame}{\leavevmode\hbox to3em{\hrulefill}\thinspace}
\providecommand{\MR}{\relax\ifhmode\unskip\space\fi MR }
% \MRhref is called by the amsart/book/proc definition of \MR.
\providecommand{\MRhref}[2]{%
  \href{http://www.ams.org/mathscinet-getitem?mr=#1}{#2}
}
\providecommand{\href}[2]{#2}
\begin{thebibliography}{1}

\bibitem{garey2002computers}
M.R. Garey and D.S. Johnson, \emph{Computers and intractability}, vol.~29,
  W.H.\ Freeman New York, 2002.

\bibitem{johnson1954optimal}
S.M. Johnson, \emph{Optimal two-and three-stage production schedules with setup
  times included}, Naval research logistics quarterly \textbf{1} (1954),
  61--68.

\bibitem{laribi2016heuristics}
I.~Laribi, F.~Yalaoui, F.~Belkaid, and Z.~Sari, \emph{Heuristics for solving
  flow shop scheduling problem under resources constraints}, IFAC-PapersOnLine
  \textbf{49} (2016), 1478--1483.

\bibitem{lenstra1977complexity}
J.K. Lenstra, A.H.G. {Rinnooy Kan}, and P.~Brucker, \emph{Complexity of machine
  scheduling problems}, Annals of discrete mathematics, vol.~1, Elsevier, 1977,
  pp.~343--362.

\bibitem{mcnaughton1959scheduling}
R.~McNaughton, \emph{Scheduling with deadlines and loss functions}, Management
  Science \textbf{6} (1959), 1--12.

\bibitem{mori1997genetic}
M.~Mori and C.C. Tseng, \emph{A genetic algorithm for multi-mode resource
  constrained project scheduling problem}, European Journal of Operational
  Research \textbf{100} (1997), 134--141.

\bibitem{stinson1978multiple}
J.P. Stinson, E.W. Davis, and B.M. Khumawala, \emph{Multiple
  resource--constrained scheduling using branch and bound}, AIIE Transactions
  \textbf{10} (1978), 252--259.

\bibitem{xing2000parallel}
W.~Xing and J.~Zhang, \emph{Parallel machine scheduling with splitting jobs},
  Discrete Applied Mathematics \textbf{103} (2000), 259--269.

\end{thebibliography}
\bibliographystyle{amsplain}

\vfill

\small
\vskip2mm plus 1fill
\noindent
Version \today{}
\bigbreak

\noindent
Felix Bock
{\tt <felix.bock@uni-ulm.de>}\\
Henning Bruhn
{\tt <henning.bruhn@uni-ulm.de>}\\
Institut f\"ur Optimierung und Operations Research\\
Universit\"at Ulm\\
Germany\\

\end{document}